\documentclass[11pt]{article}
\usepackage{graphicx}
\usepackage{amsmath}
\usepackage{amsthm}
\usepackage{amssymb}
\usepackage{mathrsfs}
\usepackage[T1]{fontenc}
\usepackage[english]{babel}
\usepackage{varioref}
\usepackage{calc}
\usepackage{ifthen}
\usepackage{enumerate}
\usepackage{paralist}
\usepackage{graphics} 
\usepackage{tikz}
\usepackage{fullpage}
\usetikzlibrary{calc,through,backgrounds,arrows}
\usepackage{authblk} 

\usepackage{color}
\definecolor{webblue}{rgb}{0, 0, 1.0}  
\definecolor{webgreen}{rgb}{0,1.0,0} 
\definecolor{webred}{rgb}{1.0, 0, 0}   

\definecolor{lily1}{rgb}{0.5,0,0.5}
\definecolor{lily2}{rgb}{0.75,0,0.25}
\definecolor{lily3}{rgb}{0.25,0,0.75}
\definecolor{green1}{rgb}{0.25,0.25,0}
\definecolor{green2}{rgb}{0.5,0.5,0}
\definecolor{green3}{rgb}{0.75,0.75,0}

\newcommand{\F}{\mathbb{F}}

\newcommand{\vek}[1]{\boldsymbol{#1}} 

\newtheorem{theorem}{Theorem}
\newtheorem{proposition}[theorem]{Proposition}
\newtheorem{lemma}[theorem]{Lemma}

\newtheorem{corollary}[theorem]{Corollary}

\newtheorem{remark}[theorem]{Remark}

\DeclareMathOperator{\PG}{PG}

\DeclareMathOperator{\Aut}{Aut}

\providecommand{\keywords}[1]
{
  \small	
  \textbf{Keywords} #1
}

\providecommand{\subclass}[1]
{
  \small	
  \textbf{Mathematics Subject Classification (2000)} #1
}

\begin{document}

\title{The Geometry of $(t\mod{q})$-arcs}

\author[1]{Sascha Kurz}
\author[2]{Ivan Landjev}
\author[3]{Francesco Pavese}
\author[4]{Assia Rousseva}
\affil[1]{Mathematisches Institut, Universit\"at Bayreuth, D-95440 Bayreuth, Germany, sascha.kurz@uni-bayreuth.de}
\affil[2]{New Bulgarian University, 21 Montevideo str., 1618 Sofia, Bulgaria and Bulgarian Academy of Sciences, Institute of Mathematics and Informatics, 
  8 Acad G. Bonchev str., 1113 Sofia, Bulgaria, i.landjev@nbu.bg}
\affil[3]{Dipartimento di Meccanica, Matematica e Management, Politecnico di Bari, Via Orabona 4, 70125, Bari, Italy, francesco.pavese@poliba.it}
\affil[4]{Faculty of Mathematics and Informatics, Sofia University, 5 J. Bourchier blvd., 1164 Sofia, Bulgaria, assia@fmi.uni-sofia.bg}

\date{}

\maketitle

\begin{abstract}
In this paper, we give a geometric construction 
of the three strong non-lifted  $(3\mod{5})$-arcs in $\PG(3,5)$
of respective sizes 128, 143, and 168, and construct an infinite family 
of non-lifted, strong $(t\mod{q})$-arcs in $\PG(r,q)$  with $t=(q+1)/2$
for all $r\ge3$ and all odd prime powers $q$. 	

\medskip

\noindent 
\keywords{$(t\mod q)$-arcs \and linear codes \and quadrics \and caps \and quasidivisible arcs \and sets of type $(m,n)$} 

\noindent 
\subclass{51E22 \quad 51E21 \quad 94B05}
\end{abstract}

\section{Introduction}\label{sec:intro}

The strong $(t\mod{q})$-arcs were introduced 
and investigated in \cite{KLR22,LR13,LR19,LRS16} in connection with the 
extendability problem for Griesmer arcs. This problem is related in turn to the problem 
of the existence and extendability of arcs associated with Griesmer codes. 
In \cite{KLR22} the classification
of the strong $(3\mod5)$-arcs was used to rule out the 
existence of the hypothetical $[104,4,82]_5$-code, one of the four undecided cases
for  codes of dimension 4 over $\mathbb{F}_5$. It turns out that apart from the many strong 
$(3\mod5)$-arcs obtained from the canonical lifting construction, there exist
three non-lifted strong $(3\mod5)$-arcs of respective sizes 128, 143, and 168.
This is a counterexample
to the conjectured impossibility of strong $(3\mod5)$-arcs in geometries over
$\mathbb{F}_5$ in dimensions larger than 2.
The three arcs are constructed by a computer search, but display regularities
which suggest a nice geometric structure.

In this paper, we give a geometric, computer-free construction of the three
non-lifted strong $(3\mod5)$-arcs in $\PG(3,5)$. Two of them are related to
the non-degenerate quadrics of $\PG(3,5)$. Their construction can be generalized further
to larger fields and larger dimensions.

\section{Preliminaries}

We define an arc in $\PG(r,q)$ as a mapping from the point set $\mathcal{P}$ 
of the geometry to the non-negative integers:
$\mathcal{K}\colon\mathcal{P}\to\mathbb{N}_0$.
An arc $\mathcal{K}$ in $\PG(r,q)$ is called a $(t\mod{q})$-arc if $\mathcal{K}(L)\equiv t\pmod{q}$
for every line $L$. It is immediate that $\mathcal{K}(S)\equiv t\pmod{q}$ for every subspace
$S$ with $\dim S\ge1$. Increasing the multiplicity of an arbitrary point by $q$ preserves the property of
being a $(t\mod{q})$-arc. So, we can assume that the point multiplicities are integers
contained in the interval $[0,q-1]$. If the maximal point multiplicity is at most $t$ we call $\mathcal{K}$ a 
\emph{strong} $(t\mod{q})$-arc.  

The extendability of the so-called $t$-quasidivisible arcs is related to structure properties of 
$(t\mod{q})$-arcs. In particular, an $(n,s)$-arc $\mathcal{K}$ in $\PG(r,q)$ with spectrum $(a_i)$
is called
\emph{$t$-quasidivisible} with divisor $\Delta$ if $s\equiv n+t\pmod{\Delta}$ and $a_i=0$ for all
$i\not\equiv n,n+1,\ldots,n+t\pmod{\Delta}$. It is quite common in coding theory that hypothetical
Griesmer codes are associated with arcs that turn out to be $t$-quasidivisible with divisor $q$
for some $t$.
The extendability of $t$-quasidivisible arcs is related to the structure of 
particular strong $(t\mod{q})$-arcs associated with them.

Let $\mathcal{K}$ be an arc in $\PG(r,q)$ and let $\sigma:\mathbb{N}_0\to\mathbb{Q}$ be a function
satisfying $\sigma(\mathcal{K}(H))\in\mathbb{N}_0$ for every hyperplane $H$ in $\mathcal{H}$,
where $\mathcal{H}$ is the set of all hyperplanes in $\PG(r,q)$. The arc
$\mathcal{K}^{\sigma}:\mathcal{H}\to\mathbb{N}_0$, $H\to\sigma(\mathcal{K}(H))$ is called the $\sigma$-dual of
$\mathcal{K}$. For a $t$-quasidivisible arc with divisor $q$, we consider the $\sigma$-dual arc
obtained for $\mathcal{K}^{\sigma}(H)=n+t-\mathcal{K}(H)\pmod{q}$.
It turns out that with this $\sigma$ the $\sigma$-dual to a $t$-quasidivisible arc $\mathcal{K}$ is 
a strong $(t\mod{q})$-arc. Moreover,
if $\mathcal{K}^{\sigma}$ contains a hyperplane in its support then 
$\mathcal{K}$ is extendable \cite{LR13,LRS16}.

There exist several straightforward constructions of $(t\mod{q})$-arcs \cite{LR13,LR19,LRS16}.
The first is the so-called sum-of-arcs construction.

\begin{theorem}
\label{thm:sum-of-arcs}	
  Let $\mathcal{K}$ and $\mathcal{K}'$ be a $(t_1 \mod q)$- and a $(t_2 \mod q)$-arc in $\PG(r,q)$, 
  respectively. Then $\mathcal{K}+\mathcal{K}'$ is a $(t \mod q)$-arc with 
$t \equiv t_1 + t_2 \pmod q$. Similarly, $\alpha\mathcal{K}$, where $\alpha\in\{ 0,\dots,\dots, p-1\}$ 
and $p$ is the characteristic of $\F_q$, is a $(t \mod q)$-arc 
with $t\equiv \alpha t_1\pmod q$.
\end{theorem}

For the special case of $t=0$, and $q=p$ we have that the sum of two $(0\mod{p})$-arcs and the scalar multiple of a 
$(0\mod{p})$-arc are again $(0\mod{p})$-arcs. Hence the set of all $(0\mod{p})$-arcs
is a vector space over $\mathbb{F}_p$, cf.~\cite{LR19}.

The second construction is the so-called \emph{lifting construction}, see \cite[p. 230]{LR19}.
\begin{theorem}
\label{thm:lifting-construction}
Let $\mathcal{K}_0$ be a (strong) $(t\mod q)$-arc in a projective $s$-space $\Sigma$ of $\PG(r, q)$, where $1 \le s < r$. For a fixed projective $(r-s-1)$-space $\Gamma$ of $\PG(r, q)$, disjoint from $\Sigma$, let $\mathcal{K}$ be the arc in $\PG(r, q)$ defined as follows: 
\begin{itemize}
\item for each point $P$ of $\Gamma$, set $\mathcal{K}(P)=t$;
\item for each point $Q \in \PG(r, q) \setminus \Gamma$, set $\mathcal{K}(Q)=\mathcal{K}_0(R)$, where $R=\langle \Gamma, Q \rangle \cap \Sigma$. 
\end{itemize}  
Then $\mathcal{K}$ is a (strong) $(t\mod q)$-arc in $\PG(r,q)$ of cardinality $q^{r-s} \cdot |\mathcal{K}_0| + t \frac{q^{r-s}-1}{q-1}$. 
\end{theorem}


Arcs obtained by the lifting construction are called \emph{lifted arcs}. If $\Sigma$ is a point, then we speak of a \emph{lifting point}. The iterative 
application of the lifting constructions gives the more general version stated above. In the other direction, in \cite[Lemma 1]{LR19} i was shown that the 
set of all lifting points forms a subspace.

The classification of strong $(t\mod{q})$ arcs in $\PG(2,q)$ is equivalent to that 
of certain plane blocking sets \cite{LR16}.

\begin{theorem}
	\label{thm:blocking-set-constr}
	A strong $(t\mod q)$-arc $\mathcal{K}$ in $\PG(2,q)$ of cardinality $mq+t$ exists if and only if there exists an 
	$((m-t)q+m,\ge m-t)$-blocking set $\mathcal{B}$ 
	with line multiplicities contained in the set $\{m-t,m-t+1,\dots,m\}$.
\end{theorem}

The condition that the multiplicity of each point is at most $t$ turns out to be very strong.
For $t=0$, we have that the only strong $(0\mod{q})$-arc is the trivial zero-arc. 
For $t=1$ the strong $(1\mod{q})$-arcs are the hyperplanes.
For $t=2$ all strong $(2\mod{q})$ arcs in $\PG(r,q)$, for $r\ge3, q\ge5$,
turn out to be lifted \cite{LR19}. In $\PG(2,q)$,  
all $(2\mod{q})$-arcs are also known (cf.~\cite[Lemma 3.7]{KLR22}).
Apart from one sporadic example, all such arcs are again lifted.
It was conjectured in \cite{LR16} that all strong $(3\mod{5})$-arcs in 
$\PG(r,5)$, $r\ge3$, are lifted.
The  computer classification reported in \cite{KLR22} shows that this conjecture is wrong:
there exist $(3\mod{5})$-arcs of respective sizes 128, 143, and 168 that are not lifted.
In the next sections we give a geometric (computer-free) description of these arcs 
and define an infinite class of strong $(t\mod{q})$-arcs in $\PG(r,q)$, $r\ge3$,
that are not lifted.

 \section{The arc of size 128}
\label{sec:128}
 
 We shall need the classification of all strong $(3\mod{5})$-arcs in $\PG(2,5)$ of sizes 18, 23, 28 and 33.
 It is obtained easily from Theorem~\ref{thm:blocking-set-constr} and can be found in \cite{KLR22,LR19}.  
 
\begin{theorem}
	\label{thm:plane-arcs}
Let $\mathcal{K}$ be a strong $(3\mod5)$-arc in $\PG(2,5)$.
Let $\lambda_i$, $i=0,1,2,3$, denote the number of $i$ points of
$\mathcal{K}$.

\begin{enumerate}[(a)]
\item If $|\mathcal{K}|=18$ then $\mathcal{K}$ is the sum of three lines.
\item If $|\mathcal{K}|=23$ then it has $\lambda_3=3,\lambda_2=4,\lambda_1=6$.
The four 2-points form a quadrangle, the three 3-points are the 
diagonal points of the quadrangle, and the 1-points are the intersections of the diagonals with the sides
of the quadrangle.
\item If $|\mathcal{K}|=28$ then it has $\lambda_3=6,\lambda_1=10$. The 3-points form an oval, and the 
1-points are the internal points to this oval.
\item There exist ten non-isomorphic arcs with $|\mathcal{K}|=33$. These are:
\begin{enumerate}[(i)]
\item the duals of the complements of the seven $(10,3)$-arcs in $\PG(2,5)$ (cf. \cite{L96});
\item the dual of the multiset which is complement of the $(11,3)$-arc with four external lines
plus one point which is not on a 6-line ($\lambda_3=6,\lambda_2=5,\lambda_1=5$);
\item the dual of a blocking set in which one double point forms an oval with five of the 0-points;
the tangent to the oval in the 2-point is a 3-line ($\lambda_3=6,\lambda_2=5,\lambda_1=5$); 
\item the modulo 5 sum of three non-concurrent lines: two of them are lines of 3-points and one 
is a line of 2-points ($\lambda_3=8,\lambda_2=4,\lambda_1=1$).
\end{enumerate}
\end{enumerate}
\end{theorem}
 
Let us note that one of the strong $(3\mod5)$-arcs in case $(d(i))$ is obtained by
taking as 3-points the points of an oval and as 1-points the external points to the oval.

Consider a $(3\mod5)$-arc $\mathcal{K}$ in $\PG(3,5)$ which is of multiplicity 128.
Let $\varphi$ be a projection from an arbitrary 0-point $P$ to a plane $\pi$ not incident with $P$:
\begin{equation}
\label{eq:project}
\varphi\colon
\left\{\begin{array}{lll}
\mathcal{P}\setminus\{P\}\ & \rightarrow\ & \pi \\
Q & \rightarrow & \pi\cap\langle P,Q\rangle .
\end{array}\right.
\end{equation}
Here $\mathcal{P}$ is again the set of points of $\PG(3,5)$.
Note that $\varphi$ maps the lines through $P$ into
points from $\pi$, and the planes through $P$ into lines in $\pi$.
For every set of points $\mathcal{F}\subset\pi$, define the induced
arc $\mathcal{K}^{\varphi}$ by
\[\mathcal{K}^{\varphi}(\mathcal{F})\,=
\,\sum_{\varphi(P)\in\mathcal{F}} \mathcal{K}(P).\]
It is clear that $P$ is incident with 3- and 8-lines, only. If there exists a 13-line
$L$ through $P$ then all planes through $L$ have multiplicity at least 33 (Theorem~\ref{thm:plane-arcs})
and $|\mathcal{K}|\ge6\cdot33-5\cdot13=133$, a contradiction.

An 8-line through $P$ is either of type (3,3,1,1,0,0) (type ($\alpha$),
or of type (3,2,2,1,0,0) (type ($\beta$)). Other types for an 8-line are
impossible by the same counting argument as above:
a plane through such a line has to be of multiplicity at least 33
(18-planes are impossible since $P$ is a 0-point), 
and we get a contradiction by the same counting argument as above.  
A 3-line through $P$ is  of type $(\gamma_1)$ (3,0,0,0,0,0),
$(\gamma_2)$ (2,1,0,0,0,0), or $(\gamma_3)$ (1,1,1,0,0,0).
A point in the projection plane is said to be of type
($\alpha$), ($\beta$), or ($\gamma_i$) if it is the image of a line of the same type.
Let us note that type $(\alpha)$ and $(\beta)$ are the same as types $(B_2)$ and $(B_3)$
from \cite{KLR22}; similarly type $(\gamma_i)$ 
coincides with type $(A_i)$, $i=1,2,3$.

By Theorem~\ref{thm:plane-arcs}, if a line in the projection plane
has one 8-point then it contains:

- one point of type ($\alpha$), one point of type ($\gamma_1$), and
four points of type ($\gamma_2$), or else

- one point of type ($\beta$), two points of type ($\gamma_1$), two points of type
($\gamma_2$) and one point of type ($\gamma_3$).

We are going to prove that if $\mathcal{K}$ is a strong $(3\mod5)$-arc in $\PG(3,5)$
of cardinality 128
then the induced arc $\mathcal{K}^{\varphi}$ in $\PG(2,5)$
is unique (up to isomorphism).
It consists of seven 8-points and 24 3-points.
Three of the 8-points are of type $(\alpha)$, and four are of type ($\beta$). 
The 3-points are: six of type ($\gamma_1$), twelve of type ($\gamma_2$), and
six of type ($\gamma_3$). The points of type ($\beta$) form a quadrangle, and 
the points of type ($\alpha$) are the diagonal points. The intersections of the 
lines defined by the diagonal points  with the sides of the quadrangle are points of type
($\gamma_3$); the six points on the lines  defined by the diagonal points that
are not on sides of the quadrangle are of type ($\gamma_1$); all the remaining 3-points are
of type ($\gamma_2$). The induced arc $\mathcal{K}^{\varphi}$ is presented on the picture below.

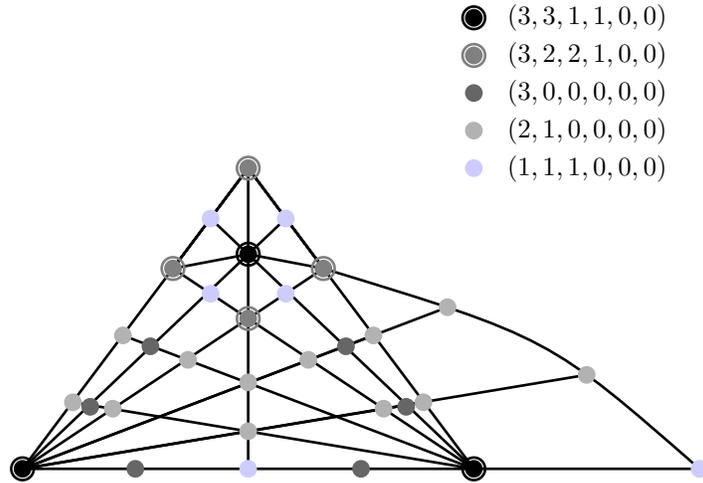
\begin{center}
	\begin{tikzpicture}[line width=1pt, scale=0.5]
	draw[gray] (0,0)
	\draw (0,0)--(0,8);
	\draw (-6,0)--(12,0);
	\draw (-6,0)--(2,5.33)--(0,8)--(-2,5.33)--(6,0)--(0,8)--(-6,0);
	\draw (-6,0)--(1,6.66);
	\draw (6,0)--(-1,6.66);
	\draw (-6,0)--(3.33,3.55);
	\draw (-6,0)--(4.66,1.77);
	\draw (6,0)--(-3.33,3.55);
	\draw (6,0)--(-4.66,1.77);
	\draw (-6,0)--(5.3,4.30);
	\draw (-6,0)--(9,2.5);
	
	\draw[black] (-2,5.33)--(0,5.71)--(2,5.33)
	.. controls (8,3.5) and (8,3.5) .. (12,0);
	
	\draw[black] (-6,0) circle (2mm) [fill=black];
	\draw[blue!20] (0,0) circle (2mm) [fill=blue!20];
	\draw[black] (6,0) circle (2mm) [fill=black];
	\draw[black!50] (0,8) circle (2mm) [fill=black!50];
	\draw[black!50] (2,5.33) circle (2mm) [fill=black!50];
	\draw[black!50] (-2,5.33) circle (2mm) [fill=black!50];
	\draw[black!50] (0,4) circle (2mm) [fill=black!50];	
	
	\draw[black] (-6,0) circle (3mm);
	\draw[black] (6,0) circle (3mm);
	\draw[black!50] (0,8) circle (3mm);
	\draw[black!50] (2,5.33) circle (3mm);
	\draw[black!50] (-2,5.33) circle (3mm);
	\draw[black!50] (0,4) circle (3mm);	
	
	\draw[black!60] (-3,0) circle (2mm) [fill=black!60];
	\draw[black!60] (3,0) circle (2mm) [fill=black!60];
	\draw[blue!20] (12,0) circle (2mm) [fill=blue!20];
	\draw[black!30] (-4.66,1.77) circle (2mm) [fill=black!30];
	\draw[black!30] (-3.33,3.55) circle (2mm) [fill=black!30];
	\draw[blue!20] (-1,6.66) circle (2mm) [fill=blue!20];
	\draw[black!30] (4.66,1.77) circle (2mm) [fill=black!30];
	\draw[black!30] (3.33,3.55) circle (2mm) [fill=black!30];
	\draw[blue!20] (1,6.66) circle (2mm) [fill=blue!20];
	\draw[black] (0,5.71) circle (2mm) [fill=black];	
	\draw[black] (0,5.71) circle (3mm);
	\draw[blue!20] (-1,4.66) circle (2mm) [fill=blue!20];
	\draw[blue!20] (1,4.66) circle (2mm) [fill=blue!20];		
	\draw[black!30] (0,1) circle (2mm) [fill=black!30];
	\draw[black!30] (0,2.3) circle (2mm) [fill=black!30];		
	\draw[black!30] (-1.6,2.9) circle (2mm) [fill=black!30];
	\draw[black!30] (-3.6,1.6) circle (2mm) [fill=black!30];		
	\draw[black!60] (-2.6,3.26) circle (2mm) [fill=black!60];
	\draw[black!60] (-4.2,1.65) circle (2mm) [fill=black!60];		
	
	\draw[black!30] (1.6,2.9) circle (2mm) [fill=black!30];
	\draw[black!30] (3.6,1.6) circle (2mm) [fill=black!30];		
	\draw[black!60] (2.6,3.26) circle (2mm) [fill=black!60];
	\draw[black!60] (4.2,1.65) circle (2mm) [fill=black!60];		
	\draw[black!30] (5.3,4.3) circle (2mm) [fill=black!30];
	\draw[black!30] (9,2.5) circle (2mm) [fill=black!30];		
	
	\draw[black] (6,12) circle (2mm) [fill=black];	
	\draw[black] (6,12) circle (3mm);
	\draw[black!50] (6,11) circle (2mm) [fill=black!50];	
	\draw[black!50] (6,11) circle (3mm);
	\draw[black!60] (6,10) circle (2mm) [fill=black!60];	
	\draw[black!30] (6,9) circle (2mm) [fill=black!30];	
	\draw[blue!20] (6,8) circle (2mm) [fill=blue!20];	
	\draw (9,12) node{\small{$(3,3,1,1,0,0)$}};
	\draw (9,11) node{\small{$(3,2,2,1,0,0)$}};
	\draw (9,10) node{\small{$(3,0,0,0,0,0)$}};
	\draw (9,9) node{\small{$(2,1,0,0,0,0)$}};
	\draw (9,8) node{\small{$(1,1,1,0,0,0)$}};

	\end{tikzpicture}
\end{center}

\begin{lemma}
\label{lma:projection-from-0-point}
Let $\mathcal{K}$ be a strong $(3\mod5)$-arc in $\PG(3,5)$ of cardinality
128. Let $\varphi$ be the projection from an arbitrary 0-point
in $\PG(3,5)$ into a plane disjoint from that point. 
Then the arc $\mathcal{K}^{\varphi}$ is unique
up to isomorphism and has the structure described above. 
\end{lemma}

 \begin{proof}
 We have seen that 0-points are incident only with lines of multiplicity 3 and 8. 
 Hence $\mathcal{K}^{\varphi}$ has seven 8-points and twenty-four 3-points.
 Assume that six of the 8-points are collinear. Clearly, every 8-point is of type ($\alpha$)
 since it is on a line containing two 8-points (and hence the image of a 28-plane).
 Every other point in the projection plane is also on a line containing
 two 8-points; hence all 3- points in the plane are of type
 ($\gamma_1$) or ($\gamma_3$). But now a line with one 8-point
 cannot have points of type ($\gamma_2$), which is a contradiction
 with the structure of the $(3\mod5)$-arc of size 23.
 
 Assume that five of the 8-points are collinear. Let $L$ be the line that
 contains them. If the two 8-points off $L$ define a line meeting $L$ in a 3-point, the proof
 is completed as above. Otherwise, the points off $L$ are on four lines containing
 two 8-points. Now it is easily checked that
 there exists a line with exactly one 8-point which has at least four 3-points that are not
 of type ($\gamma_2$). This is a contradiction with the structure of the $(3\mod5)$-arc
 of size 23.
 
 A similar argument rules out the possibility of four collinear 8-points. 
 In all cases these have to be points of type $(\alpha)$. So are the remaining
 three 8-points. Now for all possible configurations of these seven points we get a 23-line 
 without enough points of type $(\gamma_2)$.
 
 We are going to consider in full detail the case when at most three 8-points
 in the projection plane are collinear.
 Assume there exists an oval of 8-points, $X_1,\ldots,X_6$, say, and let $Y$ be the seventh
 8-point. All 8-points are of type ($\alpha$) and let $YX_1X_2$ be a 
 secant to the oval through $Y$. The lines $X_1X_j$, $j=3,4,5,6$, are images of planes without 2-points.
 Now an external line to the oval through $Y$ is a 23-line and
 has at most one point of type ($\gamma_2$), a contradiction. In a similar way, we rule out the case
where there exist five 8-points no three of which are collinear. We have to consider the different
possibilities for the line defined by the remaining two 8-points: secant, tangent, or
external line to the oval formed by the former five points and one additional point
which has to be a 3-point.

We have shown so far that there are at most three collinear 8-points.
It is also clear that there exist at least two lines that contain three 8-points.
We consider the case where these  lines meet in a 3-point.
Denote the 8-points by $X_i, Y_i$, $i=1,2,3$, and $Z$. We also assume that
$X_1,X_2,X_3$ are collinear and so are $Y_1,Y_2,Y_3$. Each of the lines $ZX_i$, $i=1,2,3$, also
contains three 8-points; otherwise there exist five 8-points no three of which are collinear.
Without loss of generality, the triples $Z,X_i,Y_i$, $i=1,2,3$, are collinear.
Now it is clear that all the points $X_i, Y_i$ are of type ($\alpha$). Moreover, neither of the 
lines $X_iY_j$, $i\ne j$, has 3-points of type ($\gamma_2$). Now if we consider a line through $X_3$ that does not
have other 8-points it should contain four points of type $(\gamma_2)$. 
On the other hand, it intersects $X_1Y_2$ and $X_1Y_3$ in points which are not of this type
which gives a contradiction.

Now we are left with only one possibility for the 8-points 
subject to the conditions: (i) each line contains at most three 8-points, (ii) 
lines incident with three 8-points meet in an 8-point, (iii) every 5-tuple of 8-points
contains a collinear triple. The 8-points are the vertices of a quadrangle plus the 
three diagonal points. Furthermore,
the diagonal points have to be of type ($\alpha$) while the vertices of the quadrangle are forced to be 
of type ($\beta$). This is due to the fact that through each of the vertices of the quadrangle there is
a line with a single 8-point which meets the three lines defined by the diagonal points of type $(\alpha)$
in three different 3-points that are not of type ($\gamma_2$) (since a 28-plane does not have 2-points). 
Thus we get the picture below.
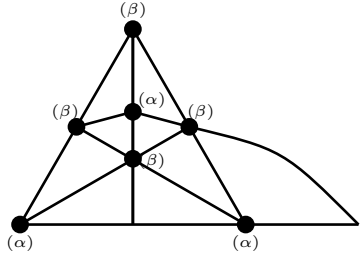
\vfill

\begin{center}
\begin{tikzpicture}[line width=1pt, scale=0.5]

\draw (0,0)--(6,0)--(3,5.2)--(0,0);
\draw (3,0)--(3,5.2);
\draw (0,0)--(4.5,2.6);
\draw (6,0)--(1.5,2.6);
\draw (3,0)--(3,5.2);
\draw (6,0)--(9,0);
\draw (1.5,2.6)--(3,3)--(4.5,2.6) .. controls (7,2) and (7,2) .. (9,0);
\draw[black] (0,0) circle (2mm) [fill=black];
\draw[black] (6,0) circle (2mm) [fill=black];
\draw[black] (3,5.2) circle (2mm) [fill=black];
\draw[black] (1.5,2.6) circle (2mm) [fill=black];
\draw[black] (4.5,2.6) circle (2mm) [fill=black];
\draw[black] (3,1.75) circle (2mm) [fill=black];
\draw[black] (3,3) circle (2mm) [fill=black];

\draw (0,-0.5) node{\tiny{$(\alpha)$}};
\draw (6,-0.5) node{\tiny{$(\alpha)$}};
\draw (3.5,3.3) node{\tiny{$(\alpha)$}};
\draw (1.2,3) node{\tiny{$(\beta)$}};
\draw (4.8,3) node{\tiny{$(\beta)$}};
\draw (3,5.7) node{\tiny{$(\beta)$}};
\draw (3.5,1.7) node{\tiny{$(\beta)$}};
\end{tikzpicture}\end{center} 

The fact that a 23-line through a point of type $(\alpha)$
contains four points of type $(\gamma_2)$ and one point of type $(\gamma_1)$
identifies the six points of type $(\gamma_1)$.

\begin{center}
	\begin{tikzpicture}[line width=1pt, scale=0.5]
	\draw[gray] (0,0)--(6.5,2);
	\draw[gray] (0,0)--(7.8,1.2);
	\draw[black] (6,0)--(3,3);
	\draw[black] (0,0)--(3,3);
	\draw (0,0)--(6,0)--(3,5.2)--(0,0);
	\draw (3,0)--(3,5.2);
	\draw (0,0)--(4.5,2.6);
	\draw (6,0)--(1.5,2.6);
	\draw (3,0)--(3,5.2);
	\draw (6,0)--(9,0);
	\draw (1.5,2.6)--(3,3)--(4.5,2.6) .. controls (7,2) and (7,2) .. (9,0);
	\draw[black] (0,0) circle (2mm) [fill=black];
	\draw[black] (6,0) circle (2mm) [fill=black];
	
	\draw[gray,dotted] (4.8,5)--(4.6,1.6);
	\draw[gray,dotted] (5.2,-1.6)--(5.2,0.4);
	\draw[gray,dotted] (5.0,-1.6)--(4.5,-0.3);
	\draw[gray,dotted] (4.8,-1.6)--(1.5,-0.3);	
	
	\draw[gray,dotted] (-1,2.5)--(0.5,1.1);
	\draw[gray,dotted] (-1,2.7)--(1.2,1.5);	
	
	\draw[black] (3,3) circle (2mm) [fill=black];
	
	\draw[black] (6.5,2) circle (1.5mm) [fill=white];
	\draw[black] (7.8,1.2) circle (1.5mm) [fill=white];
    \draw[black] (3,0.45) circle (1.5mm) [fill=white];
    \draw[black] (3,0.9) circle (1.5mm) [fill=white];	
 \draw[black] (3.9,1.2) circle (1.5mm) [fill=white];
 \draw[black] (4.6,1.4) circle (1.5mm) [fill=gray];
 \draw[black] (5.1,1.6) circle (1.5mm) [fill=white];
 
 \draw[black] (4.7,0.7) circle (1.5mm) [fill=white];	
\draw[black] (5.2,0.8) circle (1.5mm) [fill=gray];
\draw[black] (5.6,0.85) circle (1.5mm) [fill=white];
	
\draw[black] (1.4,1.4) circle (1.5mm) [fill=gray];  
\draw[black] (0.8,0.8) circle (1.5mm) [fill=gray];    
\draw[black] (1.5,0) circle (1.5mm) [fill=gray];  
\draw[black] (4.5,0) circle (1.5mm) [fill=gray];

	\draw (0,-0.5) node{\tiny{$(\alpha)$}};
	\draw (6,-0.5) node{\tiny{$(\alpha)$}};
	\draw (3.5,3.3) node{\tiny{$(\alpha)$}};

	\draw (4.8,5.3) node{\tiny{$(\gamma_1)$}};
	\draw (5.2,-2.4) node{\tiny{$(\gamma_1)$}};
	\draw (-1.7,2.6) node{\tiny{$(\gamma_1)$}};
	\end{tikzpicture}\end{center} 

Furthermore, a line with two points of type $(\alpha)$ must contain also two points of type 
$(\gamma_1)$ and two points of type $(\gamma_3)$. This identifies the six 3-points of
type $(\gamma_1)$.  The remaining 3-points are all of type $(\gamma_2)$. This implies the suggested structure.
\end{proof}

Lemma ~\ref{lma:projection-from-0-point} implies that given a nonlifted, strong
$(3\mod5)$-arc $\mathcal{K}$ of cardinality 128, every 0-point is incident with

\begin{tabular}{l}
	- three 8-lines of type $(3,3,1,1,0,0)$, \\
	- four 8-lines of type $(3,2,2,1,0,0)$, \\
	- six 3-lines of type $(3,0,0,0,0,0)$, \\
	- twelve 3-lines of type $(2,1,0,0,0,0)$, \\
	- six 3-lines of type $(1,1,1,0,0,0)$
\end{tabular} 

\noindent 
Now this implies that 

- $\#$(3-points) $=3\cdot2+4\cdot1+6\cdot1 = 16$,

- $\#$(2-points) $=4\cdot2+12\cdot1 = 20$,

- $\#$(1-points) $=3\cdot2+4\cdot1+12\cdot1+6\cdot3 = 40$,

- $\#$(0-points) $=1+3\cdot1+4\cdot1+6\cdot4+12\cdot3+2\cdot6 = 80$.

 Furthermore, each 0-point is incident with
 six 33-planes, three 28-planes eighteen 23-planes and four 18-planes. 
 Moreover the number of zeros in a 33-plane is 12, in a 28-plane -- 15, 
 in a 23-plane -- 18, and in an 18-plane -- 16.
 This makes it possible to compute the spectrum of $\mathcal{K}$. We have
 \begin{eqnarray*}
a_{33} &=& \frac{80\cdot6}{12}=40 \\
a_{28} &=& \frac{80\cdot3}{15}=16, \\
a_{23} &=& \frac{80\cdot18}{18}=80, \\
a_{18} &=& \frac{80\cdot4}{16}=20.
\end{eqnarray*}
Furthermore, every 33-, 28-, 23-plane in $\mathcal{K}$ is unique up
to isomorphism. 
 
 From the above considerations we can deduce that no three 2-points
 are collinear. In other words they form a 20-cap $C$. Moreover, this
cap has spectrum: $a_6(C)=40, a_4(C)=80, a_3(C)=20, a_0(C)=16$. 
It is not extendable to the elliptic quadric; in such case it would have 
(at least 20) tangent planes. Thus, this cap is complete and isomorphic to one of the 
two caps $K_1$ and $K_2$ by Abatangelo, Korchamros and Larato \cite{ACL96}. 
It is not $K_2$ since it has a different spectrum (cf. \cite{ACL96}).
Hence the 20-cap on the 2-points in $\PG(3,5)$ is isomorphic to $K_1$.
 
 Consider the complete cap $K_1$. The collineation group $G$ of $K_1$ is a semidirect product 
 of  an elementary abelian group
 of order 16 and a group isomorphic to $S_5$ \cite{ACL96}. Hence $|G|=1920$. 
 The action of $G$ on $\PG(3,5)$ splits the point set of $\PG(3,5)$ into four orbits on points,
 denoted by $O_1^P,\ldots,O_4^P$, and the set of lines into six orbits, denoted by
 $O_1^L,\ldots,O_6^L$. The respective sizes of these orbits are
 \[|O_1^P|=40, |O_2^P|=80, |O_3^P|=20, |O_4^P|=16;\]
 \[|O_1^L|=160, |O_2^L|=240, |O_3^L|=30, |O_4^L|=160, |O_5^L|=120, |O_6^L|=96.\]
 The corresponding point-by-line orbit matrix $A=(a_{ij})_{4\times6}$, where
 $a_{ij}$ is the number of the points from the $i$-th point orbit incident with
 any line from the $j$-th line orbit is the following
 
 \[A=\left(\begin{array}{cccccc}
  3 & 1 & 4 & 1 & 2 & 0 \\
  3 & 4 & 0 & 2 & 2 & 5 \\
  0 & 1 & 2 & 2 & 0 & 0 \\ 
  0 & 0 & 0 & 1 & 2 & 1
 \end{array}
 \right).\]
Set $w=(w_1,w_2,w_3,w_4)$. We look for solutions of the equation $wA\equiv3\vek{j}\pmod{5}$,
where $\vek{j}$ is the all-one vector, subject to the conditions $w_i\le3$ for all
$i=1,2,3,4$. The set of all solutions is given by
\begin{multline*}
\{w=(w_1,w_2,w_3,w_4)\mid w_i\{0,\ldots4\}, \\
w_2\equiv1-w_1\pmod{5}, w_3\equiv4-2w_1\pmod{5}, w_4=3\}.
\end{multline*}
There exist two solutions that satisfy $w_i\le3$:
$w=(3,3,3,3)$ and $w=(1,0,2,3)$. The first one yields the trivial $(3\mod5)$-arc 
formed by three copies of the whole space. The second one gives the desired arc of size 128.

It should be noted that the weight vectors $(0,3,2,4)$, $(1,2,0,4)$, $(2,1,3,4)$,
and $(3,0,1,4)$ yield strong $(4\mod 5)$-arcs of cardinalities 344, 264, 284, and 204, respectively,
that are not lifted. 

\section{Strong $\left(\frac{q+1}{2} \mod{q}\right)$-arcs from quadrics and the arcs of size 143 and 168}

For an arbitrary odd prime power $q$ and an integer $r\ge 2$, let $\mathcal Q$ be a quadric of $\PG(r, q)$ and let $F$ be the quadratic form defining $\mathcal Q$. This means that a point $P(x_0,\ldots,x_{r})$ of $\PG(r, q^2)$ belongs to $\mathcal Q$ whenever $F(x_0,\ldots, x_{r})=0$. The points of $\PG(r, q)$ outside $\mathcal Q$ are partitioned into two point classes, say $\mathcal{P}_1$ and $\mathcal{P}_2$. Indeed, if $P(x_0, \ldots, x_{r})$ is a point of $\PG(r, q) \setminus \mathcal Q$, then $P$ belongs to $\mathcal{P}_1$ or $\mathcal{P}_2$, according as $F(x_0,\ldots, x_{r})$ is a non-square or a square in $\F_q$. Now we define the arcs $\mathcal{K}_1$ and $\mathcal{K}_2$ in the following way:

\begin{enumerate}[$\bullet$]
\item $\mathcal{K}_1$: for a point $P$ of $\PG(r, q)$ set
\begin{equation}
\label{eq:F1}
\mathcal{K}_1(P)=\left\{
\begin{array}{cl}
\frac{q+1}{2} & \text{ if } P\in\mathcal{Q}, \\
1 & \text{ if } P \in \mathcal{P}_1, \\
0 & \text{ if } P \in \mathcal{P}_2. 
\end{array}\right.
\end{equation}

\item $\mathcal{K}_2$: for a point $P$ of $\PG(r, q)$ set
\begin{equation}
\label{eq:F2}
\mathcal{K}_2(P)=\left\{
\begin{array}{cl}
\frac{q+1}{2} & \text{ if } P\in\mathcal{Q}, \\
0 & \text{ if } P \in \mathcal{P}_1, \\
1 & \text{ if } P \in \mathcal{P}_2. 
\end{array}
\right.\end{equation}
\end{enumerate}
 
\noindent
The following result is well-known.
 
\begin{proposition}\label{prop}
 	Let $f(x)=ax^2+bx+c$, where $a,b,c,\in\mathbb{F}_q$, $a\ne0$, $q$ odd. 
 	If $\mathbb{F}_q=\{\alpha_0,\alpha_1,\ldots,\alpha_{q-1}\}$.
 	Denote by $S$ be the list of the following elements from $\mathbb{F}_q$:
 	\[ a, f(\alpha_0),f(\alpha_1),\ldots,f(\alpha_{q-1}).\]
 	Then
 	\begin{enumerate}[(a)]
 	\item if $f(x)$ has two distinct roots in $\mathbb{F}_q$ the list 
 	$S$ contains two zeros, $(q-1)/2$ squares and $(q-1)/2$ non-squares;
 	\item if $f(x)$ has one double root in $\mathbb{F}_q$
 	then $S$ contains a zero and $q$ squares, or a zero and 
 	$q$ non-squares;
 	\item if $f(x)$ is irreducible over $\mathbb{F}_q$ then $S$ contains
 	$(q+1)/2$ squares and $(q+1)/2$ non-squares.
 	\end{enumerate}
\end{proposition} 
 
\begin{theorem}
\label{thm:quadrics}
Let the $\mathcal{K}_1$ and $\mathcal{K}_2$ be the arcs defined in (\ref{eq:F1}) and (\ref{eq:F2}), respectively. Then $\mathcal{K}_i$ is a $\left(\frac{q+1}{2} \mod{q}\right)$ arc of $\PG(r, q)$, $i = 1, 2$. Moreover, if $\mathcal{Q}$ is non-degenerate, then both arcs are not lifted.
\end{theorem}
\begin{proof}
Let $\ell$ be a line of $\PG(r, q)$, then $\mathcal{Q} \cap \ell$ is a quadric of $\ell$. Then, from Proposition~\ref{prop}, it follows that 
\begin{align*}
\mathcal{K}_i(\ell) = 
\begin{cases}
2 \cdot \frac{q+1}{2} + \frac{q-1}{2} & \mbox{ if } |\ell \cap \mathcal{Q}| = 2, \\
\frac{q+1}{2} + q & \mbox{ if } |\ell \cap \mathcal{Q}| = 1 \mbox{ and } |\ell \cap \mathcal{P}_i| = q, \\
\frac{q+1}{2} & \mbox{ if } |\ell \cap \mathcal{Q}| = 1 \mbox{ and } |\ell \cap \mathcal{P}_i| = 0, \\
\frac{q+1}{2} & \mbox{ if } |\ell \cap \mathcal{Q}| = 0. 
\end{cases}
\end{align*}
Therefore $\mathcal{K}_i$ is a $\left(\frac{q+1}{2} \mod{q}\right)$ arc of $\PG(r, q)$, $i = 1, 2$. If $\mathcal{Q}$ is non-degenerate, then through every point of $\PG(r, q)$ there exists a line $r$ that is secant to $\mathcal{Q}$. By construction, the line $r$ has two $\frac{q+1}{2}$-points, $\frac{q-1}{2}$ $1$-points and $\frac{q-1}{2}$ $0$-points. Hence $\mathcal{K}_i$ is not lifted.   
\end{proof}

\begin{corollary}
If $r$ is odd, then 
\begin{align*}
|\mathcal{K}_i| =  
\begin{cases}
\frac{q+1}{2} \cdot \frac{( q^{\frac{r+1}{2}} + 1)(q^{\frac{r-1}{2}} - 1)}{q-1} + \frac{q^r + q^{\frac{r-1}{2}}}{2} & \mbox{ if } \mathcal{Q} \mbox{ is elliptic},\\  
\frac{q+1}{2} \cdot \frac{( q^{\frac{r-1}{2}} + 1)(q^{\frac{r+1}{2}} - 1)}{q-1} + \frac{q^r - q^{\frac{r-1}{2}}}{2} & \mbox{ if } \mathcal{Q} \mbox{ is hyperbolic}.
\end{cases} \\
\end{align*}
If $r$ is even, then
\begin{align*}
|\mathcal{K}_1| = \frac{q+1}{2} \cdot \frac{(q^r - 1)}{q-1} + \frac{q^r - q^{\frac{r}{2}}}{2}, \\
|\mathcal{K}_2| = \frac{q+1}{2} \cdot \frac{(q^r - 1)}{q-1} + \frac{q^r + q^{\frac{r}{2}}}{2}.   
\end{align*}
\end{corollary}

\begin{remark}
In the case when the quadric $\mathcal{Q}$ is degenerate, then it is not difficult to see that the arc $\mathcal{K}_i$, $i = 1, 2$, is lifted. Let $\mathcal{Q}$ be a non-degenerate quadric of $\PG(r, q)$, then $\mathcal{K}_1$ and $\mathcal{K}_2$ are projectively equivalent if $r$ is odd, but they are not in the case when $r$ is even. On the other hand, if $r$ is odd, there are two distinct classes of non-degenerate quadrics, namely the hyperbolic quadric and the elliptic quadric. Therefore in all cases Theorem~\ref{thm:quadrics} gives rise to two distinct examples of non lifted $\left(\frac{q+1}{2} \mod{q}\right)$ arcs of $\PG(r, q)$.
\end{remark}

\subsection{The arcs of size 143 and 168}
 
 In \cite{KLR22}, the following two strong non-lifted $(3\mod5)$-arcs in $\PG(3,5)$
 were constructed by a computer search. The respective spectra are:
 
 
\[|\mathcal{F}_1|=143,\ \ \ a_{18}(\mathcal{F}_1)=26, a_{23}(\mathcal{F}_1)=0,
a_{28}(\mathcal{F}_1)_{28}=65,a_{33}(\mathcal{F}_1)=65;\]

\[\lambda_0(\mathcal{F}_1)=65,\lambda_1(\mathcal{F}_1)=65,\lambda_2(\mathcal{F}_1)=0,
\lambda_3(\mathcal{F}_1)=26,\]
and
 
 \[|\mathcal{F}_2|=168,\ \ \ a_{28}(\mathcal{F}_2)=60,a_{33}(\mathcal{F}_2)=60,
 a_{43}(\mathcal{F}_2)36);\]

\[\lambda_0(\mathcal{F}_2)=60,\lambda_1(\mathcal{F}_2)=60,\lambda_2(\mathcal{F}_2)=0,
\lambda_3(\mathcal{F}_2)=36.\]
In addition, $|\Aut(\mathcal{F}_1)|=62400$, and $|\Aut(\mathcal{F}_2)|=57600$.
 
 
These arcs can be recovered from Theorem~\ref{thm:quadrics}. Indeed, if $\mathcal{Q}$ is an elliptic quadric of $\PG(3,5)$, then $\mathcal{K}_1$ is a non lifted $(3 \mod{5})$ arc of $\PG(3, 5)$ of size $143$, whereas if $\mathcal{Q}$ is a hyperbolic quadric of $\PG(3, 5)$, then $\mathcal{K}_1$ is a non lifted $(3 \mod{5})$ arc of $\PG(3, 5)$ of size $168$. 

\section{Further examples $(t \mod q)$-arcs}

A set of type $(m, n)$ in $\PG(r, q)$ is a set $\mathcal{S}$ of points such that every line of $\PG(r, q)$ contains either $m$ or $n$ points of $\mathcal{S}$, 
$m < n$, and both values occur. Assume $m > 0$. Then the only sets of type $(m, n)$ that are known, exist in $\PG(2, q)$, $q$ square, and are such that $n = m + \sqrt{q}$. 
In particular, sets of type $(1,1+ \sqrt{q})$ either contains $q+ \sqrt{q} + 1$ and are Baer subplanes or $q\sqrt{q} + 1$ points and are known as {\em unitals}. For more details 
on sets of type $(m, n)$ in $\PG(2, q)$ see \cite{PR} and references therein. If $\mathcal{S}$ is an $(m, n)$ set in $\PG(r, q)$, $r > 2$, then necessarily $q$ is an odd square, 
$m = (\sqrt{q} - 1)^2/2$, $n = m + \sqrt{q}$ and $|\mathcal{S}| = \frac{1 + \frac{q^r-1}{q-1}(q - \sqrt{q}) \pm \sqrt{q}^r}{2}$, see \cite{TS}. However no such a set is known 
to exist if $r > 2$. 

\begin{theorem}
Let $\mathcal{S}$ be a set of type $(m, m + \sqrt{q})$ in $\PG(r, q)$, $q$ square. Let $\mathcal{K}$ be the arc of $\PG(r, q)$ such that $\mathcal{K}(P) = \sqrt{q}$, if $P \in \mathcal{S}$ and $\mathcal{K}(P) = 0$, if $P \notin \mathcal{S}$. Then $\mathcal{K}$ is a $(\sqrt{q} \mod q)$-arc of $\PG(r, q)$.
\end{theorem}
\begin{proof}
Let $\ell$ be a line of $\PG(r, q)$. If $|\ell \cap \mathcal{S}| = m$, then $\mathcal{K}(\ell) = m \sqrt{q}$, whereas if $|\ell \cap \mathcal{S}| = m + \sqrt{q}$, then $\mathcal{K}(\ell) = m \sqrt{q} + q$. 
\end{proof}

In $\PG(r, q)$, $q$ square, let $\mathcal{H}$ be a Hermitian variety of $\PG(r, q)$, i.e., the variety defined by a Hermitian form of $\PG(r, q)$. It is well-known that a line of $\PG(r, q)$ has $1$, $\sqrt{q}+1$ or $q+1$ points in common with $\mathcal{H}$. Let $\mathcal{K'}$ be the arc of $\PG(r, q)$ such that $\mathcal{K'}(P) = \sqrt{q}$, if $P \in \mathcal{H}$ and $\mathcal{K'}(P) = 0$, if $P \notin \mathcal{H}$. 

\begin{theorem}
$\mathcal{K'}$ is a $(\sqrt{q} \mod q)$-arc of $\PG(r, q)$. Moreover, if $\mathcal{H}$ is non-degenerate, then $\mathcal{K'}$ is not lifted.
\end{theorem}
\begin{proof}
Let $\ell$ be a line of $\PG(r, q)$. Then 
\begin{align*}
\mathcal{K'}(\ell) =
\begin{cases}
\sqrt{q} & \mbox{ if } |\ell \cap \mathcal{H}| = 1, \\
\sqrt{q} + q & \mbox{ if } |\ell \cap \mathcal{H}| = \sqrt{q} + 1, \\ 
\sqrt{q}(1 + q) & \mbox{ if } |\ell \cap \mathcal{H}| = q + 1.
\end{cases}
\end{align*}
If $\mathcal{H}$ is non-degenerate, then through every point of $\PG(r, q)$ there exists a line $r$ such that $|\mathcal{H} \cap r| = \sqrt{q}+1$. By construction, the line $r$ has $\sqrt{q}+1$ $\sqrt{q}$-points and $q-\sqrt{q}$ $0$-points. Hence $\mathcal{K'}$ is not lifted.   
\end{proof}

\section*{Acknowledgements}
The  research of the second author 
was  supported  by  the  Bulgarian  National  Science  Research  Fund
under  Contract  KP-06-Russia/33.
The research of the fourth author was supported  by  the  Research  Fund  of
Sofia  University  under  Contract  80-10-52/10.05. 2022. 
All authors  would like to thank the organizers of the Sixth Irsee Conference on Finite Geometries for their invitation. During that conference 
the main ideas for this paper emerged.


\end{document}